\documentclass[reqno,a4paper]{amsart}
\usepackage[utf8]{inputenc}
\usepackage[T1]{fontenc}
\usepackage[british]{babel}
\usepackage{amsfonts, amssymb, amsmath}
\usepackage[all]{xy}
\usepackage[colorlinks, citecolor=blue]{hyperref}
\usepackage[noabbrev, capitalise,nameinlink]{cleveref} 
\usepackage{mathbbol}
\usepackage{calrsfs}
\usepackage{mathabx}
\usepackage{amsfonts}
\usepackage{amssymb}
\usepackage{amsthm}
\usepackage{amsmath}
\allowdisplaybreaks
\usepackage{xcolor}
\definecolor{l-gray}{gray}{0.7}

\theoremstyle{plain}
\newtheorem{theorem}{Theorem}[section]
\newtheorem{proposition}[theorem]{Proposition}
\AddToHook{env/proposition/begin}{\crefalias{theorem}{proposition}}
\newtheorem{lemma}[theorem]{Lemma}
\AddToHook{env/lemma/begin}{\crefalias{theorem}{lemma}}
\newtheorem{corollary}[theorem]{Corollary}
\AddToHook{env/corollary/begin}{\crefalias{theorem}{corollary}}

\theoremstyle{definition}
\newtheorem{definition}[theorem]{Definition}
\AddToHook{env/definition/begin}{\crefalias{theorem}{definition}}

\AddToHook{env/definitions/begin}{\crefalias{theorem}{definitions}}

\theoremstyle{remark}
\newtheorem{remark}[theorem]{Remark}
\AddToHook{env/remark/begin}{\crefalias{theorem}{remark}}

\AddToHook{env/remarks/begin}{\crefalias{theorem}{remarks}}
\newtheorem{example}[theorem]{Example}
\AddToHook{env/example/begin}{\crefalias{theorem}{example}}

\numberwithin{equation}{section}
\numberwithin{figure}{section}

\newenvironment{tfae}
{
	\begin{enumerate}}
	{\end{enumerate}}

\newcommand{\noproof}{\hfill \qed}

\newcommand{\univ}[2]{(#1, #2)}

\newcommand{\defn}{\textbf}

\DeclareMathOperator{\Ker}{Ker}

\newcommand{\C}{\ensuremath{\mathcal{C}}}
\newcommand{\D}{\ensuremath{\mathcal{D}}}
\newcommand{\E}{\ensuremath{\mathcal{E}}}

\newcommand{\V}{\mathcal{V}}

\newcommand{\Arr}{\ensuremath{\mathsf{Arr}}}

\newcommand{\XMod}{\ensuremath{\mathsf{XMod}}}

\newcommand{\Set}{\mathsf{Set}}

\newcommand{\Ch}{\ensuremath{\mathsf{Ch}}}

\renewcommand{\P}{{\rm (P)}}

\newcommand{\G}{\mathbb{G}}
\newcommand{\s}{\ensuremath{\mathsf{s}}}

\newcommand{\NN}{\mathbb{N}}

\newcommand{\X}{\mathcal{X}}

\newcommand{\Le}{\mathcal{L}}
\newcommand{\Left}{\mathrm{L}}

\DeclareMathOperator{\Coker}{Coker}
\DeclareMathOperator{\Coeq}{Coeq}

\DeclareMathOperator{\Homology}{H}
\renewcommand{\H}{\Homology}

\DeclareMathOperator{\N}{N}

\newdir{>>}{{}*!/3.5pt/:(1,-.2)@^{>}*!/3.5pt/:(1,+.2)@_{>}*!/7pt/:(1,-.2)@^{>}*!/7pt/:(1,+.2)@_{>}} 
\newdir{ >>}{{}*!/8pt/@{|}*!/3.5pt/:(1,-.2)@^{>}*!/3.5pt/:(1,+.2)@_{>}}
\newdir{ |>}{{}*!/-3.5pt/@{|}*!/-8pt/:(1,-.2)@^{>}*!/-8pt/:(1,+.2)@_{>}} 
\newdir{ >}{{}*!/-8pt/@{>}}
\newdir{>}{{}*:(1,-.2)@^{>}*:(1,+.2)@_{>}}
\newdir{<}{{}*:(1,+.2)@^{<}*:(1,-.2)@_{<}}

\def\pullback{
	\ar@{-}[]+R+<6pt,-1pt>;[]+RD+<6pt,-6pt>%
	\ar@{-}[]+D+<1pt,-6pt>;[]+RD+<6pt,-6pt>}

\def\reversepullback{ 
	\ar@{-}[]+R+<6pt,-1pt>;[]+RU+<6pt,6pt>%
	\ar@{-}[]+U+<-1pt,6pt>;[]+RU+<6pt,6pt>}

\def\dottedpullback{%
	\ar@{.}[]+R+<6pt,-1pt>;[]+RD+<6pt,-6pt>%
	\ar@{.}[]+D+<1pt,-6pt>;[]+RD+<6pt,-6pt>}

\def\pushout{%
	\ar@{-}[]+L+<-6pt,1pt>;[]+LU+<-6pt,6pt>%
	\ar@{-}[]+U+<-1pt,6pt>;[]+LU+<-6pt,6pt>}

\usepackage{tikz-cd}

\DeclareMathOperator{\coker}{coker}

\usepackage{hyperref}
\hypersetup{
	colorlinks=true,       
	linkcolor=blue,          
	citecolor=magenta,        
	filecolor=green,      
	urlcolor=cyan           
}

\allowdisplaybreaks

\begin{document}
\emergencystretch 3em
	
\title{The first \v Cech derived functor}
	
\author{Maxime Culot}
\author{Tomas Everaert}
	
\email{maxime.culot@uclouvain.be}
\email{tomas.everaert@vub.be}
	
\address[Maxime Culot]{Institut de Recherche en Math\'ematique et Physique, Universit\'e catholique de Louvain, che\-min du cyclotron~2 bte~L7.01.02, B--1348 Louvain-la-Neuve, Belgium}
	
\address[Tomas Everaert]{Mathematics \& Data Science, Vrije Universiteit Brussel, Pleinlaan 2, B--1050 Brussel, Belgium}
	
	
\begin{abstract}
	Let $\C$ be a semi-abelian category, $\D$ a Birkhoff subcategory of $\C$, and $I$ the associated reflector. For $C\in \C$, we establish an isomorphism between $\H_1(\N(I(\check{C}(f,C))))$ (where $\N$ is the Moore normalization functor and $\check{C}(f,C)$ is a \v Cech resolution associated to $C$) and $\H_1(N(I(\G(C))))$ (where $\G(C)$ is the canonical $\G$-simplicial resolution of $C$ for $\G$ the comonad derived from an adjunction $F\colon \X\to \C\dashv U\colon \C\to \X$ satisfying classical conditions introduced by Quillen). The former corresponds to the Hopf formula for the second homology object associated to $I$, while the latter represents the first \v Cech derived functor of $I$. Specifically, we show that both can be expressed by a single unified formula.
		
	Furthermore, we can also prove that, under certain assumptions, the first \v Cech derived functor of $I$ is naturally isomorphic to the first non-additive derived functor of $I$ (in the sense of Culot--Renaud--Van~der~Linden).
\end{abstract}
	
	
\maketitle
	
\tableofcontents

\section*{Introduction}
	
Abelian categories have historically served as the foundational framework for defining homotopical structures and computing the homology of positively graded chain complexes. Notably, the Dold–Kan–Puppe equivalence~\cite{DoldcorrespondenceRMod,Dold-Puppe} establishes that any positively graded chain complex can be interpreted as the image of a simplicial object under the Moore normalization functor. Conversely, a simplicial object can be systematically constructed from any given positively graded chain complex, as demonstrated in textbooks such as~\cite{Weibel}. 
	
The lack of a similar equivalence has long posed challenges for extending such methods to non-additive categories, exemplified by the category of groups. Here, simplicial methods provide a robust alternative to chain complexes. For instance, Hopf’s formula for groups (for the case $n=2$ see~\cite{Hopf} and for its generalizations to higher orders~\cite{Brown-Ellis}) demonstrates their applicability. By using projective simplicial resolutions~\cite{Tierney-Vogel2}, such methods extend further to any semi-abelian category~\cite{EGVdL}. Concretely, computing the homology of the Moore normalization---which, in a Barr‑exact homological category, yields a proper chain complex~\cite{EverVdL2}---produces the expected results. Specifically, this homology captures exactness in cases where the homology vanishes, as detailed in~\cite{VdLinden:Doc}. Furthermore, the simplicial techniques enable the computation of derived functors, such as the Barr--Beck derived functors, which are closely related to the Hopf-type formulas obtained in~\cite{EGVdL}. Those formulas are Baer invariants~\cite{EverVdL1,EverVdL2}, meaning they are independent of the specific projective presentation chosen for a given object. In addition, under certain assumptions on the underlying categories and on the chosen functor, we can also prove that $(n+1)$-th Barr--Beck derived functor is naturally isomorphic to the $n$-th non-additive left derived functor (in the sense of~\cite{MC-FR-TVdL}).
	
In this paper, we establish a connection between the second homology of the Barr‑Beck derived functors associated with the reflector $I$ from a semi-abelian category $\C$ onto a selected Birkhoff subcategory $\D$ and the first \v Cech homology. As far as we know, this \v Cech homology has been rigorously defined, up to now, only in the context of groups~\cite{Inasaridze75,IH,Pirashvili-1979}.
	
We begin by revisiting the concept of \v Cech resolutions, as detailed in \cref{def-Cech resolution and homology}, along with the associated derived functors of $I$.  From this framework, we derive an explicit expression for the first derived functor (see \cref{proposition:Cech homology H2}). Subsequently, we compare this result with $\H_2(C,I)_\G$ (see \cref{sec - Barr beck approach H2}), where $\G$ is the comonad on $\C$ derived from an adjunction $F\dashv U$ satisfying standard assumptions (see \cref{sec : assumptions} for further details). This comparison is achieved through the application of standard simplicial techniques, as outlined in \cref{prop: Barr-Beck homology H2}.
	
Thus, we establish the desired isomorphism $\H_{2}(C,I)_{\G}\cong \Le_1(I)(C)$. Concerning the $\H_{n}(C,I)_{\G}$, for $n\geq 3$, we do not have, for the moment, a clear link with the homology of the \v Cech‑derived functor of $I$. Nevertheless, as highlighted in~\cite{EGVdL}, $\H_{2}(C,I)_{\G}$ plays a pivotal role as the foundational component for deriving—via induction—a general formula for $\H_{n}(C,I)_{\G}$ with $n\geq 2$. Thus, this work aims to provide the first critical step toward a broader result that unifies these two approaches.
	
\section{Context and assumptions}\label{sec : assumptions}
	
Throughout this text, $\C$ will denote a semi-abelian category with $\D$ a Birkhoff subcategory~\cite{Janelidze-Kelly} of $\C$ where $I$ denotes the associated reflector. Furthermore, we will consider a comonad $\G$
\begin{equation*}
	\G=(G\colon \C\to \C, \delta\colon G\implies G^2,\varepsilon\colon G\implies 1_\C)\label{eq - def of a comonad}
\end{equation*}
on $\C$ satisfying the following conditions
\begin{enumerate}
	\item $\varepsilon_C\colon G(C)\to C$ is a regular epimorphism for all $C\in \C$;
	\item for all $C\in \C$, we have that $G(C)$ is a projective object in $\C$ (with respect to regular epimorphisms).
\end{enumerate}
In other words, we require that for all $C\in \C$, $\varepsilon_C\colon G(C)\to C$ determines a \defn{projective presentation of $C$}.

To be able to deduce from these conditions some additional ones, we will, moreover, require $\G$ to be induced by an adjunction 
\begin{equation}
\xymatrix{
	\C \ar@<-1.1ex>[r]_-{U}^-{\perp}& \X \ar@<-1.1ex>[l]_-{F}
}.\tag{D}\label{eq - adjunction (D)}
\end{equation}
in the sense that $G=FU$ and $\epsilon$ is the counit of this adjunction. The above conditions then translate into the ones introduced by Quillen in~\cite{Quillen}. 

Supposing $\G$ is indeed induced by an adjunction   satisfying Quillen's conditions, then we also have the following: if $f\colon C\to C'$ is a regular epimorphism (or a $1$-fold extension), then the commutative square
\begin{equation}\label{eq - diagram we want to be 2 ext}
\xymatrix{
	G(C)\ar@<0.5ex>@{-{ >>}}[r]^-{G(f)}\ar@{-{ >>}}[d]_-{\varepsilon_C} & G(C')\ar@{{ >}->}@<.5ex>[l] \ar@{-{ >>}}[d]^-{\varepsilon_{C'}}\\
	C\ar@{-{ >>}}[r]_-{f}&C'
}
\end{equation} 
is a regular pushout~\cite{Bourn2003} (or a $2$-fold extension), i.e.\ a commutative square of $1$-fold extensions where, in addition, the comparison morphism into the pullback is also a $1$-fold extension,---this is a specific case of a result from~\cite{EGVdL}, which we will now recall. For this, note that~\eqref{eq - adjunction (D)} induces, for every $n\geq 0$, an adjunction
\begin{equation}
	\xymatrix{
		\Arr^n(\C) \ar@<-1.1ex>[r]_-{U_n}^-{\perp}& \Arr^n(\X) \ar@<-1.1ex>[l]_-{F_n}
	}.\tag{E}\label{eq - adjunction (E)}
\end{equation}
which is defined degree-wise, where $\Arr^n(\C)=\Arr(\Arr^{n-1}(\C))$ with $\Arr^1(\C)=\Arr(\C)$, and where $\Arr^n(\X)$ is defined analogously. Hence, an arrow in $\Arr^1(\C)$ is a commutative square and an arrow in $\Arr^2(\C)$ is a commutative cube. Then, for each $n$, the induced adjunction~\eqref{eq - adjunction (E)} induces a comonad $\G_n=(G_n= F_n\circ U_n, \delta_n, \varepsilon_n)$ on $\Arr^n(\C)$. Finally, we already know what is a $1$-fold extension and a $2$-fold extension. Inductively, for $n\geq 2$, a $n$-fold extension is a commutative square in $\Arr^{n-1}(\C)$ where all arrows and the factorization to the pullback are $(n-1)$-fold extension.

Proposition 5.6. in~\cite{EGVdL} now states the following: for any $n$-extension $C$, we have that $\varepsilon_{n,C}\colon G_n(C)\to C$ is a $(n+1)$-fold extension. In particular, $G_n(C)$ is a $n$-extension as soon as $C$ is a $n$-extension\label{G preserves extension}. It follows that the commutative square~\eqref{eq - diagram we want to be 2 ext} is indeed a regular pushout when $f$ is a regular epimorphism and, in particular, if we choose $f=\epsilon_C$, the square
\[
\xymatrix{
	G^2(C)\ar@<0.5ex>@{-{ >>}}[r]^-{G(\varepsilon_C)}\ar@{-{ >>}}[d]_-{\varepsilon_{G(C)}} & G(C) \ar@{-{ >>}}[d]^-{\varepsilon_C}\ar@{{ >}->}@<.5ex>[l]\\
	G(C)\ar@{-{ >>}}[r]_-{\varepsilon_C}&C
}
\]
is a regular pushout as well. 

\begin{remark}
	It is natural for us to consider the above conditions on $\G$ (or on the adjunction $U\dashv F$) since they proved crucial in~\cite{EGVdL} for establishing higher Hopf formulae---which are fundamental to the present article. Note that they are certainly satisfied whenever $\C$ is a semi-abelian category that is monadic over $\X=\Set$ and $\G$ is the induced comonad on $\C$: condition (1) then follows from Beck's monadicity Theorem (see for instance~\cite[Theorem 4.3.5.]{Borceux:Cats2}) and condition (2) from the fact that $F(X)$ is a free object and therefore projective. Hence, in particular, semi-abelian varieties (in the sense of universal algebra, as in~\cite{Bourn-Janelidze}) constitute a class of examples. Note that in this case, the Birkhoff subcategories $\D$ of $\C$ are exactly the subvarieties of $\C$.
\end{remark}

\begin{remark}[Regular pushout and left-adjoint]\label{rem - regular pushout are preserved by left-adjoint}
	Since a regular pushout in a semi-abelian category is simply a pushout of a regular epimorphism by a regular epimorphism~\cite[Theorem 5.7.]{Carboni-Kelly-Pedicchio}, the reflector $I\colon \C\to \D$---being colimit preserving---preserves them. This will be important in what follows.
\end{remark}
	
\section{The main result}\label{sec:setting and Cech homology}
With the context and assumptions established in the previous section, we are now ready to introduce our main result: 
\begin{theorem}\label{thm - theorem Barr Beck n=2 same as Cech n=1}
Let $C\in \C$, then 
\[
	\H_{2}(C,I)_{\G}\cong \Le_1(I)(C)
\]
where $\G$ is the comonad on $\C$ derived from the adjunction~\eqref{eq - adjunction (D)}.
\end{theorem}

It is necessary to define precisely what each side of the equation refers to. We begin with the left-hand side: this is the \defn{second Barr--Beck homology object} of $C$, which is constructed using the comonad $\G$. To see how, first recall that given such a comonad $\G=(G,\delta, \varepsilon)$ on $\C$, one can construct, for any object $C\in\C$, an augmented simplicial object, denoted $\G(C)$:
\[
\xymatrix{
	G^{n+1}(C)\ar@<1.5 ex>[r]^-{\partial_i \coloneq G^i(\varepsilon_{G^{n-i}(C)})}  & G^n(C)\ar@<1.5ex>[l]^-{\sigma_i\coloneq G^i(\delta_{G^{n-1-i}(C)})} & \cdots  &G^2(C)\ar@<3ex>[r]^-{G(\varepsilon_C)}\ar@<-3ex>[r]_-{\varepsilon_{G(C)}}&G(C)\ar[r]^-{\varepsilon_C}\ar[l]|{\delta_C}& C
}.
\]
The simplicial identities are straightforward consequences of the comonad axioms: $\varepsilon_{G(C)}\delta_C=1_C=G(\varepsilon_C)\delta_C$ and $\delta_{G(C)}\delta_C=G(\delta_C)\delta_C$. In addition, if $\G$ is derived from an adjunction satisfying Quillen's conditions recalled in \cref{sec : assumptions}, the simplicial object turns out to be a projective simplicial resolution in the sense of~Tierney and Vogel~\cite{Tierney-Vogel2}. In this situation, we call $\G(C)$ the \defn{canonical $\G$-simplicial resolution of $C$}. To define Barr--Beck homology, we also need the following concept:

\begin{definition}\label{def - Moore normalisation functor}
	Let $\mathbb{A}$ be a (semi-)simplicial object in a pointed category with pullbacks $\C$. The \defn{Moore chain complex} or \defn{normalized chain complex} $\N(\mathbb{A})$ is the chain complex defined by $\N_{n}(\mathbb{A})=0$ for $n<0$, $\N_0(\mathbb{A})=A_0$ and
	\[
	\N_{n}(\mathbb{A})\coloneq \Ker((\partial_{l}^\mathbb{A})_{1\leq l \leq n}\colon A_n \to A_{n-1}^n)=\bigcap_{l=1}^{n}\Ker(\partial_{l}^\mathbb{A}\colon A_n \to A_{n-1})
	\]
	when $n\geq 1$. We denote the canonical inclusion $\kappa_n^\mathbb{A}\colon \N_n(\mathbb{A})\to A_n$. The boundary morphism $d^{\N(\mathbb{A})}_n\colon \N_{n}(\mathbb{A})\to \N_{n-1}(\mathbb{A})$ is defined by restricting $\partial_{0}^\mathbb{A}\colon A_n\to A_{n-1}$ when $n\geq 1$, and is chosen to be zero when $n\leq 0$. This defines a functor $\N$ from $\s(\C)$ (the category of simplicial objects) to $\Ch^+(\C)$ (the category of positively graded chain complexes).		
	When $\mathbb{A}$ is a (quasi-)simplicial object in a pointed exact Malt'sev category---in particular in a semi-abelian category---$\N(\mathbb{A})$ defines a \emph{proper} chain complex~\cite{EverVdL2}.
\end{definition}
	
For any $C\in\C$, and any $n\geq 1$, the $n+1$-st Barr--Beck homology object of $C$ is then defined as follows:
\[
	\H_{n+1}(C,I)_\G\coloneq  \H_n(\N(I(\G(C))))
\]
where the dimension shift is there for historical reasons (see, for instance, \cite{EverVdL2, JuliaThesis} for further details). It can be shown that under the assumptions of~\cref{sec : assumptions}, the above expression is independent (up to isomorphism) of the chosen comonad on $\C$~\cite[Theorem 8.1.]{EGVdL}.

Now we turn our attention to the right-hand side of~\cref{thm - theorem Barr Beck n=2 same as Cech n=1}. For this, we need the concept of a \emph{\v Cech resolution of $C\in \C$}: let $C\in \C$ and let $f\colon P\to C$ be a projective presentation of $C$ (i.e.\ $f$ is a regular epimorphism and $P$ is a projective object). Consider the cube
\begin{equation}
	\vcenter{\xymatrix{
	P\times_C P\times_C P\ar@{-{ >>}}[rr]^-{\partial^2_2}\ar@{-{ >>}}[dd]_-{\partial^2_1} \ar@{-{ >>}}[dr]^-{\partial^2_0} &&P\times_C P\ar@{-{ >>}}[dr]^-{\partial^1_0} \ar@{-{ >>}}[dd]_-(0.2){\partial^1_1}|\hole& \\
	&P\times_C P\ar@{-{ >>}}[rr]^-(0.2){\partial^1_1}\ar@{-{ >>}}[dd]_-(0.2){\partial^1_0}&& P\ar@{-{ >>}}[dd]^-{f}\\
	P\times_C P\ar@{-{ >>}}[rr]^-(0.2){\partial^1_1}|(0.55)\hole \ar@{-{ >>}}[rd]_-{\partial^1_0} &&P\ar@{-{ >>}}[dr]^-{f}& \\
	&P\ar@{-{ >>}}[rr]_-{f}&&C
	}}\label{fig - definition of qi's}
\end{equation}
in $\C$ where all the faces are pullbacks. In particular, we have that
\begin{itemize}
	\item $(\partial^1_0,\partial^1_1)$ is the kernel pair of $f$;
	\item $(\partial^2_0,\partial^2_1)$ is the kernel pair of $\partial^1_0$;
	\item $(\partial^2_1,\partial^2_2)$ is the kernel pair of $\partial^1_1$.
\end{itemize}
By construction, the cube is a $3$-fold extension: the front square is a double extension as a pullback of the extension $f$ along itself, and the cube is a pullback (in the category $\Arr(\C)$ of morphisms of $\C$) of this double extension along itself.

Note that the pullbacks are in fact kernel pairs and that the $(\partial^2_0,\partial^2_1,\partial^2_2)$ is the \emph{simplicial kernel} of $(\partial^1_0,\partial^1_1)$. Hence, each step of the above construction can be obtained as the \emph{simplicial kernels} of the previous arrows and so on:
\begin{definition}~\cite{Tierney-Vogel2}
	Let $\left(f_i\colon X\to Y\right)_{0\leq i\leq n}$ be a sequence of $n+1$ morphisms in a category $\C$. A \defn{simplicial kernel} of $(f_i)_{0\leq i \leq n}$ is a sequence $\left(k_i\colon K\to X\right)_{0\leq i\leq n+1}$ of $n+2$ morphisms in $\C$ such that $f_ik_j=f_{j-1}k_i$ for $i<j$ and which is universal with respect to this property.
\end{definition}
Hence, iterating simplicial kernels, in a category with binary pullbacks, on the family of arrows we obtain at the previous step leads to:
\begin{definition}\cite{Donadze-Inassaridze-Porter,IH,Pirashvili-1979}\label{def-Cech resolution and homology}
A \defn{\v Cech resolution of $C$}, denoted $\check C(f,C)$, is the augmented simplicial object
\[
	\xymatrix{
			\cdots\ar@<2.25ex>[r]\ar@<-2.25ex>[r]\ar@<0.75ex>[r]\ar@<-0.75ex>[r]& P\times_C P \times_C P \ar@<1.5ex>[r]^-{\partial^2_0}\ar@<-1.5ex>[r]_-{\partial^2_2}\ar[r]|-{\partial^2_1}& P\times_C P\ar@<0.75ex>[r]^-{\partial^1_0}\ar@<-0.75ex>[r]_-{\partial^1_1}& P\ar[r]^-{f}&C
	}
\]
where $f\colon P\to C$ is a projective presentation of $C$. Hence $\check C(f,C)_0=P$, $\check C(f,C)_1=P\times_C P$ and, for $n\geq 2$, 
\[
	\check C(f,C)_n\coloneq \underbrace{P\times_C \cdots \times_C P}_{(n+1)\text{ factors}}
\] 
is the $n$-fold pullback.
\end{definition}
		
\begin{example}[Category of sets]\label{ex - Cech resolution in set}
In the category of sets, we can understand the above construction as follows: for $n\geq 0$
\[
	\xymatrix{
		\check{C}(f,C)_n\coloneq \underbrace{P\times_C \cdots \times_C P}_{(n+1)-\text{factors}}=\{(x_0,\ldots, x_{n})\in P^{n+1}\mid f(x_0)=\cdots = f(x_{n})\}
	}
\]
and for face operators $\partial_{i}^n\colon \check{C}(f,C)_n\to \check{C}(f,C)_{n-1}$ and for the degeneracy operators $\sigma_{i}^n\colon \check{C}(f,C)_n\to \check{C}(f,C)_{n+1}$:
\begin{align*}
	&\partial_i^n(x_0,\cdots, x_n)=(x_0,\ldots, x_{i-1},x_{i+1}, \ldots, x_n)\\ 
	&\sigma_{i}^n(x_0,\ldots, x_n)=(x_0,\ldots, x_i, x_i,x_{i+1},\ldots, x_n)
\end{align*}
for $0\leq i\leq n$ (see~\cite{IH} for further details).
\end{example}

The following Proposition is an adaptation of Tierney and Vogel’s Theorem 2.4.~\cite{Tierney-Vogel2}:

\begin{proposition}[Lifting and uniqueness up to simplicial equivalence]\label{prop- Cech resolution unique up to homtopy}
	Any morphism $x\colon C\to C'$ lifts to a morphism of \v Cech resolutions $(f_n)\colon \check{C}(f,C)\to \check{C}(g,C')$, where $f$ and $g$ denote, respectively, a projective presentation of $C$ and $C'$. 
	
	If $(f_n),(g_n)\colon \check{C}(f,C)\to \check{C}(g,C')$ denote two possible liftings of $x\colon C\to C'$ then there exists a simplicial homotopy from $(f_n)$ to $(g_n)$. 
	
	In particular, for any object $C$, when the underlying category has finite limits and enough projectives, there exists a \v Cech resolution, unique up to a simplicial homotopy equivalence.
\end{proposition}
\begin{proof}
	Tierney and Vogel’s proof can easily be adapted to work also in our situation with Cech resolutions
\end{proof}

Let's also recall the standard result relating simplicial homotopies and homology:
\begin{lemma}
	Let $f,g\colon \mathbb{A}\to \mathbb{B}$ be simplicial morphisms in a semi-abelian category. If there exists a simplicial homotopy from $f$ to $g$ then we have $\H_n(\N(f))\cong\H_n(\N(g))$ for all $n\geq 0$.
\end{lemma}
\begin{proof}
	See~\cite{GVdL} for the original proof and~\cite{MC-FR-TVdL} for an alternative and more ``recent'' proof.
\end{proof}
Combining the two previous results leads to:
\begin{corollary}
	Let $F\colon \C\to \E$ be a functor from a category $\C$ with finite limits and to $\E$ a semi-abelian category. For any $n\in \mathbb{N}$, we define 
	\[
	\Le_n(F)(C)\coloneq \H_n(\N(F(\check{C}(C)))).
	\]
	where $\check{C}(C)$ refers to $\check{C}(f,C)$ for $f$ any choice of projective presentation of $C$. For each $n\in \mathbb{N}$, $\Le_n(F)\colon \C\to \E$ is a functor defined up to natural isomorphism.
\end{corollary}
\begin{definition}[\v Cech derived functors]
	With the same assumptions and notations, for each $n\in \mathbb{N}$, we call $\Le_n(F)\colon \C\to \E$ the \emph{$n$-th \v Cech derived functor of $F$}.
\end{definition}

\begin{remark}[Inassaridze's \v Cech derived functors]
	When the codomain of $F$ is the category of groups, Inassaridze called the collection of groups $(\H_n(\N(F(\check{C}(f,C)))))_{n\geq 0}$ the \defn{\v Cech derived functors of $F$ with respect to $f$} (see for instance~\cite[Definition  2.3.8.]{IH}).
\end{remark}

\section{The proof}

By now we have explained our main result which concerns the $1$st \v Cech derived functor of $I$ (the reflector into a Birkhoff subcategory $\D$ of $\C$), but we still have to prove it. As a first step, we establish a formula for the right-hand side. We will use the following notation: given, in a regular category, a subobject $X$ of $Y$ and a morphism $g\colon Y\to Z$, we write $g(X)$ for the image of $X$ by $g$, i.e.\ the subobject given by the image factorization
\[
\xymatrix{
	X\ar@{->>}[r]\ar@{{ >}->}[d]& g(X)\ar@{{ >}->}[d]\\
	Y\ar[r]_-{g}& Z.
}
\]
The subobject $g(X)$ is a normal subobject of $Z$ when $X$ is a normal subobject of $Y$, $g$ is a regular epimorphism and the category is semi-abelian (see axiom (SA$^\ast$ $6$) in~\cite{Janelidze-Marki-Tholen}). Hence $g(X)$ can be understood to be the image of $X$ through $g$. 
A last notation: when $Y$ is a normal subobject of $X$, we denote by $X/Y$ the cokernel associated to the inclusion of $X$ into $Y$.

\subsection{First \v Cech derived functor of \texorpdfstring{$I$}{I}}
	
\begin{proposition}\label{proposition:Cech homology H2}
	$\Le_1(I)(f,C)\cong \frac{\Ker(I(\partial^1_0))\cap \Ker(I(\partial^1_1))}{I(\partial_0^2)(\Ker(I(\partial^2_1))\cap \Ker(I(\partial^2_2)))}$
\end{proposition}

\begin{proof}
In this proof—and throughout this text—given a morphism g, when we use the notation $\underline{g}$ it will always denote a morphism which is induced by $g$ in an obvious way.

The idea of the proof is to understand what are the objects $\N_n\coloneq \N(I(\check{C}(f,C)))_n$ and the morphisms $d_n\coloneq d_n^{\N(I(\check{C}(f,C)))}$ for $n=-1$, $0$, $1$ and $2$ of the Moore normalization chain complex. With the notations of \cref{def - Moore normalisation functor} and of \cref{def-Cech resolution and homology}, we can consider 

\begin{equation}\label{fig - Moore complex of Cech resolution}
		\vcenter{\resizebox{\textwidth}{!}{
		\xymatrix{		
			\cdots\ar@<2.25ex>[d]\ar@<-2.25ex>[d]\ar@<0.75ex>[d]\ar@<-0.75ex>[d] &\cdots\ar[d]\\
			I(P\times_C P\times_C P)\ar@<3.25ex>[dd]^-{I(\partial_0^2)}\ar@<-3.25ex>[dd]_-{I(\partial_2^2)}\ar[dd]|-{I(\partial_1^2)}&\N_2=\Ker(I(\partial^2_1))\cap \Ker(I(\partial^2_2))\ar[dd]_-{d_2}\ar@{-->}[dr]^-{\underline{d_2}}\ar@{{ |>}->}[l]\ar@{-{ >>}}[r]&\mathrm{Im}(\underline{d_2})\ar@{{ |>}->}[d]^-{m}\\
			&&\Ker(d_1)\ar@{{ |>}->}[dl]^-{\ker(d_1)}\ar@{-{ >>}}[d]^-{\coker(\underline{d_2})}\\
			I(P\times_C P)\ar@<0.75ex>[d]^-{I(\partial_0^1)}\ar@<-0.75ex>[d]_-{I(\partial_1^1)}&\N_1=\Ker(I(\partial^1_1))\ar[d]|-{d_1=I(\partial^1_0)\ker(I(\partial^1_1))}\ar@{{ |>}->}[l]&\H_1(\N)\\
			I(P)&\N_0= I(P)\ar[d]\ar@{=}[l]&\\
			&\N_{-1}= 0&
		}
		}
		}
\end{equation}
By pullback concatenation, it can be shown that $\Ker(I(\partial_1^0))\cap \Ker(I(\partial_1^1))$ is the kernel of $d_1$. Moreover, the cokernel of $\underline{d_2}$ is isomorphic to the cokernel of $m$, the image of $d_2'$ which is normal since $\N$ a proper complex. To conclude, it suffices then to prove that the image of $\underline{d_2}$ is $I(\partial_0^2)(\Ker(I(\partial^2_1))\cap \Ker(I(\partial^2_2)))$:
\[
\mathrm{Im}(\underline{d_2})\cong \mathrm{Im}(d_2)\cong \mathrm{Im}(I(\partial_0^2))\cong I(\partial_0^2)(\Ker(I(\partial^2_1))\cap \Ker(I(\partial^2_2))).
\] 
\end{proof}
	
\subsection{Barr--Beck simplical approach to \texorpdfstring{$\H_2$}{H2}}\label{sec - Barr beck approach H2}

The purpose of this subsection is to prove the following result---which, combined with Proposition \ref{proposition:Cech homology H2} will imply our main result:
\begin{proposition}\label{prop: Barr-Beck homology H2}\label{prop - classical result to compute H0}
	$\H_{2}(C,I)_{\G}\cong \frac{\Ker(I(\partial^1_0))\cap \Ker(I(\partial^1_1))}{I(\partial_0^2)(\Ker(I(\partial^2_1))\cap \Ker(I(\partial^2_2)))}$
\end{proposition}
	
Consider the diagram
\begin{equation}
	\xymatrix{
		0\ar[r] &\Ker(I(\G(\partial^2_2)))\ar@{{ |>}->}[r]\ar@{-->}@<-.75ex>[d]_-{\underline{I(\G(\partial^2_0))}}\ar@{-->}@<.75ex>[d]^-{\underline{I(\G(\partial^2_1))}}&I(\G(P\times_C P \times_C P))\ar@{-{ >>}}[r]^-{I(\G(\partial^2_2))}\ar@<-.75ex>[d]_-{I(\G(\partial^2_0))}\ar@<.75ex>[d]^-{I(\G(\partial^2_1))} & I(\G(P\times_C P))\ar[r]\ar@<-.75ex>[d]_-{I(\G(\partial^2_0))}\ar@<.75ex>[d]^-{I(\G(\partial^2_1))}& 0\\
		0\ar[r] &\Ker(I(\G(\partial^1_1)))\ar@{{ |>}->}[r]\ar@{-->}[d]_-{\underline{I(\G(\partial^1_0))}}& I(\G(P\times_C P))\ar@{-{ >>}}[r]^-{I(\G(\partial^1_1))} \ar[d]_-{I(\G(\partial^1_0))}& I(\G(P))\ar[r] \ar[d]^-{I(\G(f))} & 0\\
		0\ar[r]&\Ker(I(\G(f)))\ar@{{ |>}->}[r]& I\G(P)\ar@{-{ >>}}[r]^-{I(\G(f))}& I(\G(C))\ar[r]&0
	}\label{diag: setting}
\end{equation}
in the category $\s(\D)$ of simplicial objects in $\D$, where we use the same notation as in \cref{proposition:Cech homology H2}. 

The idea of the proof is to consider the long exact homology sequence associated with the bottom row in~\eqref{diag: setting} and infer our claim from the exactness of this sequence~\cite{EverVdL2}. To do that, we will need a few results:

\begin{lemma}\label{lemma:split epi and H0}
	If $g\colon X\to Y$ is a split epimorphism then $\H_0(\Ker(I(\G(g))))\cong \Ker(I(g))$.
\end{lemma}
\begin{proof} 
	If we consider the long exact homology sequence associated with the split short exact sequence
	\[
	\xymatrix{
		0\ar[r]&\Ker(I(\G(g)))\ar@{{ |>}->}[r]& I\G(X)\ar@{-{ >>}}@<.75ex>[r]^-{I(\G(g))}& I(\G(Y))\ar[r]\ar@{{ >}->}@<.75ex>[l]&0
	}
	\]
	then we have
	\[
	\resizebox{\textwidth}{!}
	{
		\xymatrix{
			&&\Coker(H_1(I(\G(g))))=0\ar@{{ |>}->}[dr]&&&&&\\
			\H_1(I(\G(X))) \ar@{-{ >>}}@<.5ex>[r]^-{H_1(I(\G(g)))}&\H_1(I(\G(Y)))\ar@{{ >}->}@<.5ex>[l]\ar[rr]\ar@{-{ >>}}[ru]^-{\coker(H_1(I(\G(g))))}&&\H_0(\Ker(I(\G(g))))\ar[rr]\ar@{-{ >>}}[rd]_-{g}&&I(X)\ar@{-{ >>}}@<.5ex>[r]^-{I(g)}&I(Y)\ar@{{ >}->}@<.5ex>[l]\ar[r] &0\\
			&&&&\Ker(I(g))\ar@{{ |>}->}[ru]&&&
		}
	}
	\]
	where $\H_0(I(\G(X)))\cong I(X)$ and $\H_0(I(\G(Y)))\cong I(Y)$ by~\cite[Proposition 6.5.]{EverVdL2}. By the exactness at $I(X)$, $\H_0(\ker(I(\G(g))))$ factors through the kernel of $I(g)$ via a regular epimorphism. Since $\H_1(I(\G(g)))$ is a split epimorphism, it has a trivial cokernel which, by exactness in $\H_0(\Ker(I(\G(g))))$, is the kernel of $p$. Hence, $p$ is in fact an isomorphism, and this prove our claim.
\end{proof}

\begin{lemma}\label{lemma - section regular pushout and projective}
	Consider a regular pushout
	\[
	\xymatrix{
		X'\ar@{-{ >>}}[r]^-{f'}\ar@{-{ >>}}[d]_-{x}& Y'\ar@{-{ >>}}[d]^-{y}\\
		X\ar@{-{ >>}}[r]_-{f} & Y
	}.
	\]
	For any given section $t$ of $y$, if $X$ is projective then we can define a section $s$ of $x$ such that $f's=tf$.
\end{lemma}
\begin{proof}
	Let $t\colon Y\to Y'$ be a section of $y$. By the pullback property, we obtain a section $t'$ of $\pi_1$ such that $\pi_2t'=tf$.
	\[
	\xymatrix{
		X'\ar@{-{ >>}}[rr]^-{f'}\ar@{-{ >>}}[dd]_-{x}\ar@{--{ >>}}[dr]|-{\univ{x}{f'}}&& Y'\ar@{-{ >>}}[dd]_-{y}\\
		&X\times_Y Y'\pullback\ar@{-{ >>}}[ru]_-{\pi_2}\ar@{-{ >>}}[dl]^-{\pi_1}\\
		X\ar@{-{ >>}}[rr]_-{f}\ar@/^1pc/@{{ >}-->}[ur]^-{t'} \ar@/^ 1.5pc/@{{ >}-->}[uu]^ -{s}  && Y \ar@/_1.5pc/@{{ >}-->}[uu]_-{t}
	}.
	\]
	Since we have a regular pushout, the comparison morphism $\univ{x}{f'}$ is a regular epimorphism. As a consequence, since $X$ is projective, we can define a section of $x$ as a lifting of $t'$ over the regular epimorphism $\univ{x}{f'}$, so that $\univ{x}{f'}s=t'$. This morphism $s\colon X\to X'$ is indeed a section of $x$:
	\[
	xs=\pi_1 \univ{x}{f'}s= \pi_1 t'= 1_X.
	\]
	
	And it is compatible with the section $t$ of $y$:
	\[
	f's= \pi_2\univ{x}{f'}s=\pi_2 t'=tf.\qedhere
	\]
\end{proof}

\begin{proposition}\cite[Proposition 3.9., Corollary 3.10.]{EverVdL2}\label{lemma - H0 is clear}
	In a semi-abelian category, let 
	\[
	\xymatrix{
		A\ar@<0.75ex>[r]^-{\partial_0}\ar@<-0.75ex>[r]_-{\partial_1} & B\ar[r]^-{e}& C
	}
	\]
	be a fork and $t\colon B\to A$ a map with $\partial_0 t= 1_B=\partial_1 t$. The following are equivalent:
	\begin{tfae}
		\item $e$ is the coequalizer of $\partial_0$ and $\partial_1$;
		\item $e$ is the cokernel of $\partial_1 \ker(\partial_0)$.
	\end{tfae}
	In particular, given a simplicial object $\mathbb{A}$ (with face operators $\partial_0^1,\partial_1^1\colon A_1\to A_0$), we have that $\H_0(\mathbb{A})=\Coeq(\partial_0^1,\partial_1^1)$.
\end{proposition}

As a consequence, we can compute $\H_0(\Ker(I(\G(f))))$ as a coequalizer object in $\D$. However, $\Ker(I(\G(f)))$ is itself a coequalizer object in $\s(\D)$:
\begin{lemma}\label{lemma:IGp1 coequalizer}
	The morphism $\underline{I(\G(\partial^1_0))}$ is the coequalizer of $(\underline{I(\G(\partial^2_0))},\underline{I(\G(\partial^2_1))})$ in $\s(\D)$.
\end{lemma}

\begin{proof}
	By applying $\G$ to the $3$-fold extension~\eqref{fig - definition of qi's}, we obtain a cube
	\[
	\xymatrix{
		\G(P\times_C P\times_C P)\ar[rr]^-{\G(\partial^2_2)}\ar[dd]_-{\G(\partial^2_1)} \ar[dr]^-{\G(\partial^2_0)} &&\G(P\times_C P)\ar[dr]^-{\G(\partial^1_0)} \ar[dd]_-(0.2){\G(\partial^1_1)}|\hole& \\
		&\G(P\times_C P)\ar[rr]^-(0.2){\G(\partial^1_1)}\ar[dd]_-(0.2){\G(\partial^1_0)}&& \G(P)\ar[dd]^-{\G(f)}\\
		\G(P\times_C P)\ar[rr]^-(0.2){\G(\partial^1_1)}|(0.55)\hole \ar[rd]_-{\G(\partial^1_0)} &&\G(P)\ar[dr]^-{\G(f)}& \\
		&\G(P)\ar[rr]_-{\G(f)}&&\G(C)
	}
	\]
	of simplicial resolutions which consists, level-wise, of cubes in $\C$ that are threefold extensions, as follows from Proposition 5.6. in~\cite{EGVdL}---see p~\pageref{G preserves extension}. As a result, the left-hand induced commutative square
	\[
	\xymatrix@C=5em{
		\G(P\times_C P\times_C P)\ar@{-{ >>}}[d]_-{\G(\partial^2_2)}\ar@{--{ >>}}[r]^-{\univ{\G(\partial_0^2)}{\G(\partial_1^2)}}&R[\G(\partial^1_0)]\ar@{--{ >>}}[d]\ar@<0.75ex>[r]^-{\pi_1}\ar@<-0.75ex>[r]_-{\pi_2}&\G(P\times_C P)\ar[d]^-{\G(\partial_1^1)}\\
		\G(P\times_C P)\ar@{--{ >>}}[r]_-{\univ{\G(\partial_1^1)}{\G(\partial_0^1)}}&R[\G(f)]\ar@<0.75ex>[r]^-{\pi'_1}\ar@<-0.75ex>[r]_-{\pi'_2}& \G(P)
	}
	\]
	where $(R[\G(\partial^1_0)],\pi_1,\pi_2)$ and $(R[\G(f)],\pi'_1,\pi'_2)$ are, respectively, the kernel pairs of $\G(\partial^1_0)$ and $\G(f)$, consists level-wise of regular pushouts.
	
	Since $I$ preserves regular pushouts (see \cref{rem - regular pushout are preserved by left-adjoint}), also the top right-hand parallellogram in the diagram
	\[
	\resizebox{\textwidth}{!}
	{
		\xymatrix{
			\Ker(I(\G(\partial^2_2)))\ar@{{ |>}->}[rr]\ar@<-1ex>[dd]_-(0.3){\underline{I(\G(\partial^2_0))}}\ar@<1ex>[dd]^-(0.3){\underline{I(\G(\partial^2_1))}}\ar@{-->}[dr]&&I(\G(P\times_C P\times_C P))\ar[dr]\ar[rr]^-{I(\G(\partial^2_2))}\ar@<-1ex>[dd]_-(0.2){I\G(\partial^2_0)}|\hole\ar@<1ex>[dd]^-(0.2){I(\G(\partial^2_1))}|\hole&&I(\G(P\times_C P))\ar[rd]\ar@<-1ex>[dd]_-(0.2){I(\G(\partial^2_0))}|\hole\ar@<1ex>[dd]^-(0.2){I\G(\partial^2_1)}|\hole&\\
			& \Ker(\underline{I(\G(\partial^1_1))})\ar@{{ |>}->}[rr]\ar@<-1ex>[dl]_-{\underline{I(\pi_1)}}\ar@<1ex>[dl]^-{\underline{I(\pi_2)}}&&I(R[\G(\partial^1_0)])\ar[rr]_-(0.2){\underline{I(\G(\partial^1_1))}}\ar@<-1ex>[dl]_-{I(\pi_1)}\ar@<1ex>[dl]^-{I(\pi_2)}&&I(R[\G(f)])\ar@<-1ex>[dl]_-{I(\pi'_1)}\ar@<1ex>[dl]^-{I(\pi'_2)}\\
			\Ker(I(\G(\partial^1_1)))\ar@{{ |>}->}[rr]\ar[d]_-{\underline{I(\G(\partial^1_0))}}&&I\G(P\times_C P)\ar[rr]^-{I(\G(\partial^1_1))}\ar[d]_-{I(\G(\partial^1_0))}&&I(\G(P))\ar[d]^-{I(G(f))}\\
			\Ker(I(\G(f)))\ar@{{ |>}->}[rr]&&I(\G(P))\ar[rr]_-{I(\G(f))}&&I(\G(C))
		}
	}
	\]
	in $\s(\D)$ consists level-wise of regular pushouts. From this it follows that the induced morphism between $\Ker(I(\G(\partial^2_2)))$ and $\Ker(\underline{I(\G(\partial^1_1))})$ consist level-wise of regular epimorphisms (e.g.\ \cite[Lemma 1.2.7.]{Tomasthesis}).

	Hence, to finish the proof, it will suffice to show that the fork
	\[
	\xymatrix{
		\Ker(\underline{I(\G(\partial^1_1))})\ar@<0.75ex>[rr]^-{\underline{I(\pi_1)}} \ar@<-0.75ex>[rr]_-{\underline{I(\pi_2)}}&& \Ker(I(\G(\partial^1_1)))\ar[r]^-{\underline{I(\G(\partial^1_0))}}&\Ker(I(\G(f)))
	}
	\]
	is a coequalizer diagram. Actually, we will prove that it consists level-wise of split coequalizers as being the image of a split coequalizers through the kernel functor.
	
	Consider, in $\s(\C)$, the diagram
	\[
	\xymatrix{
		R[\G(\partial^1_0)]\ar[r]\ar@<0.75ex>[d]^-{\pi_1}\ar@<-0.75ex>[d]_-{\pi_2}&R[\G(f)]\ar@<0.75ex>[d]^-{\pi'_1}\ar@<-0.75ex>[d]_-{\pi'_2}\\
		\G(P\times_C P) \ar@{-{ >>}}[r]^-{\G(\partial^1_1)} 
		\ar@<0.75ex>@{-{ >>}}[d]^-{\G(\partial^1_0)} & \G(P)\ar@<0.75ex>@{-{ >>}}[d]^-{\G(f)}\\
		\G(P)  \ar@{-{ >>}}[r]_-{\G(f)}\ar@{{ >}->}@<.75ex>[u]&\G(C)\ar@{{ >}->}@<.75ex>[u]
	}
	\]
	$\G(f)$ consists level-wise of regular epimorphisms between projective objects, hence of split epimorphisms. Since taking the kernel pair of a split epimorphism yields a split coequalizer, it follows that the right-hand vertical fork in the diagram consists level-wise of split coequalizers. Furthermore, as $G$ preserves them, the bottom square consists level-wise of regular pushouts, in which the downward pointing morphisms can be endowed---by Lemma \ref{lemma - section regular pushout and projective}---with compatible splittings. As a result, also the left-hand vertical fork consists level-wise of split coequalizers, and the whole diagram is a morphism between forks which consists level-wise of morphisms of split coequalizers. Finally, if we apply $I$ to this morphism and next take the kernel, we must necessarily obtain a coequalizer---consisting level-wise of split coequalizers---as desired.
\end{proof}

\begin{corollary}\label{cor- expression of H0KerIGf}
	With the above notation, 
	\[\H_0(\Ker(I(\G(f))))\cong \frac{\Ker(I(\partial^1_1))}{I(\partial_0^2)(\Ker(I(\partial^2_1))\cap \Ker(I(\partial^2_2)))}.\]
\end{corollary}
\begin{proof}
	By~\cref{lemma - H0 is clear}, we know that all the vertical forks below determine coequalizer diagrams
	\[
	\resizebox{\textwidth}{!}{
	\xymatrix@C=5em{
	\H_0(\Ker(I(\G(\partial_2^2))))\ar@<-.75ex>[r]_-{\H_0(\underline{I(\G(\partial^2_0))})}\ar@<.75ex>[r]^-{\H_0(\underline{I(\G(\partial^2_1))})} & \H_0(\Ker(I(\G(\partial_1^1)))) \ar[r]^-{\H_0(\underline{I(\G(\partial^1_0))})} & \H_0(\Ker(I(\G(f))))\\
	\Ker(I(\G(\partial_2^2)))_0\ar@<-.75ex>[r]_-{\underline{I(\G(\partial^2_0))}_0}\ar@<.75ex>[r]^-{\underline{I(\G(\partial^2_1))}_0}\ar@{-{ >>}}[u] & \Ker(I(\G(\partial_1^1)))_0 \ar@{-{ >>}}[r]^-{\underline{I(\G(\partial^1_0))}_0}\ar@{-{ >>}}[u] & \Ker(I(\G(f)))_0\ar@{-{ >>}}[u]\\
	\Ker(I(\G(\partial_2^2)))_1\ar@<-.75ex>[r]_-{\underline{I(\G(\partial^2_0))}_1}\ar@<.75ex>[r]^-{\underline{I(\G(\partial^2_1))}_1} \ar@<.75ex>[u]^-{\partial_0^1}\ar@<-.75ex>[u]_-{\partial_1^1} & \Ker(I(\G(\partial_1^1)))_1\ar@{-{ >>}}[r]^-{\underline{I(\G(\partial^1_0))}_1} \ar@<.75ex>[u]^-{\partial_0^1}\ar@<-.75ex>[u]_-{\partial_1^1} & \Ker(I(\G(f)))_1 \ar@<.75ex>[u]^-{\partial_0^1}\ar@<-.75ex>[u]_-{\partial_1^1}
	}
	}
	\]
	On the other hand, by~\cref{lemma:IGp1 coequalizer}, the bottom and the middle forks also determine a coequalizer diagram. As a result, the top fork is a coequalizer diagram as well.
	
	Finally, since $\partial_2^2$ and $\partial_1^1$ are split epimorphisms, \cref{lemma:split epi and H0} tells us that $\H_0(\Ker(I(\G(\partial^1_1))))\cong \Ker(I(\partial^1_1))$ and $\H_0(\Ker(I(\G(\partial^2_2))))\cong \Ker(I(\partial^2_2))$. Hence, $\H_0(\Ker(I(\G(f))))=\Coeq(\underline{I(\partial^2_0)},\underline{I(\partial^2_1)})$. 
	
	\[\resizebox{\textwidth}{!}{
		\xymatrix{
			\Ker(I(\partial_1^2))\cap \Ker(I(\partial_2^2))\cong \Ker(\underline{I(\partial_1^2)}) \ar@{{ |>}->}[r]^-{\ker(\underline{I(\partial_1^2)})}\ar@{-{ >>}}[dr]  & \Ker(I(\partial_2^2)) )\ar@{-{ >>}}@<0.75ex>[r]^-{\underline{I(\partial_0^2)}} \ar@{-{ >>}}@<-0.75ex>[r]_-{\underline{I(\partial_1^2)}} &  \Ker(I(\partial_1^1)) \ar@{-{ >>}}[rr] && \H_0(\Ker(I(\G(f))))\\
			& \underline{I(\partial_0^2)}(\Ker(I(\partial_1^2))\cap \Ker(I(\partial_2^2))) \ar@{{ |>}->}[ru]&
	}}
	\]
	
	By~\cref{lemma - H0 is clear}, this last expression can be reformulated as
	\begin{align*}
		\H_0(\Ker(I(\G(f))))&=\Coker(\underline{I(\partial^2_0)} \ker(\underline{I(\partial^2_1)}))= \frac{\Ker(I(\partial^1_1))}{\underline{I(\partial_0^2)}(\Ker(I(\partial^2_1))\cap \Ker(I(\partial^2_2)))}\\
		&\cong \frac{\Ker(I(\partial^1_1))}{I(\partial_0^2)(\Ker(I(\partial^2_1))\cap \Ker(I(\partial^2_2)))}.\label{eq:H0KIGf version 1}
	\end{align*}
	where the second equality can be understood from the above diagram and the last equation follows from the proof of~\cref{proposition:Cech homology H2}.
\end{proof}

With all the above results, we are now finally ready to prove~\cref{prop - classical result to compute H0}.

First recall that since $P$ is projective, $\G(P)$ is a contractible augmented simplicial object~\cite{Barr-Beck}, hence so is $I(\G(P))$. As a result, we have that $\H_n(I(\G(P))=0$ for all $n\geq 1$ by~\cite[Proposition 3.11.]{EverVdL2}.

Then, if we consider the middle row in~\eqref{diag: setting}, we obtain the commutative diagram with exact rows:
\begin{equation}
	\resizebox{\textwidth}{!}{
	\xymatrix{
		\cdots & 0\cong\H_1(I(\G(P)))\ar[r] \ar[d] & \H_0(\Ker(I(\G(\partial_1^1))))\cong \Ker(I(\partial_1^1))\ar@{{ |>}->}[r] \ar[d]|-{\H_0(\underline{I(\G(\partial_0^1))})\cong \coker(\underline{I(\partial^2_0)} \ker(\underline{I(\partial^2_1)}))}  & \H_0(I(\G(P\times_C P)))\cong I(P\times_C P) \ar@{-{ >>}}[r] \ar[d]|-{I(\partial_0^1)}& \H_0(I(\G(P)))\cong I(P)\ar@{-{ >>}}[d]_-{I(f)}\ar[r]& 0\\
		0\cong\H_1(I(\G(P)))\ar[r]& \H_1(I(\G(C)))\ar[r]^-{\delta}& \H_0(\Ker(I(\G(f))))\ar[r]^-{\H_0(\ker(I(\G(f))))}\ar@{-{ >>}}[d]_-{q}& \H_0(I(\G(P)))\cong I(P)\ar@{-{ >>}}[r]_-{I(f)}& \H_0(I(\G(C)))\cong I(C)\ar[r] & 0\\
		&&\Ker(I(f))\ar@{{ |>}->}[ru]_-{\ker(I(f))}&&
	}
	}\label{fig - long exact sequence in homology}
\end{equation}
By~\cite[Proposition 6.5]{EverVdL2} we have isomorphisms $\H_0(I(\G(P)))\cong I(P)$, $\H_0(I(\G(C)))\cong I(C)$ and $\H_0(I(\G(P\times_C P)))\cong I(P\times_C P)$. The isomorphism $\H_0(\underline{I(\G(\partial_0^1))})\cong \coker(\underline{I(\partial^2_0)} \ker(\underline{I(\partial^2_1)}))$ follows from the proof of~\cref{cor- expression of H0KerIGf}.
	
By the exactness of the bottom sequence, we have a short exact sequence 
	\[
	\xymatrix{
		0\ar[r]&\H_1(I(\G(C)))\ar@{{ |>}->}[r]^-{\delta} & \H_0(\Ker(I(\G(f))))\ar@{-{ >>}}[r]^-{q} & \Ker(I(f))\ar[r]& 0
	}
\]
in $\D$, which fits into the commutative diagram with exact rows.
\begin{equation}
	\resizebox{\textwidth}{!}{
	\xymatrix{
	I(\partial_0^2)(\Ker(I(\partial^2_1)\cap \Ker(I(\partial^2_2))))\ar@{=}[r]\ar@{{ >}->}[d]&I(\partial_0^2)(\Ker(I(\partial^2_1)\cap \Ker(I(\partial^2_2))))\ar@{-{ >>}}[r]\ar@{{ |>}->}[d]&0\ar@{{ |>}->}[d]\\
		\Ker(I(\partial^1_0))\cap \Ker(I(\partial^1_1))\cong \Ker(\underline{I(\partial^1_0)})\ar@{{ |>}->}[r]\ar@{-->}[d]& \Ker(I(\partial^1_1))\ar@{-{ >>}}[d]|-{\coker(\underline{I(\partial^2_0)} \ker(\underline{I(\partial^2_1)}))}\ar@{-{ >>}}[r]^-{\underline{I(\partial^1_0)}}&\Ker(I(f))\ar@{=}[d]\\
		\H_1(I(\G(C)))\ar@{{ |>}->}[r]_-{\delta} & \H_0(\Ker(I(\G(f))))\ar@{-{ >>}}[r]_-{q} & \Ker(I(f))
	}
	}\label{diag - 3x3 situation for the conclusion}
\end{equation}
Note that the commutativity of the top right-hand square follows from the simplicial identities and that of the bottom right-hand square from the commutativity of diagram~\eqref{fig - long exact sequence in homology}.

From the proof of~\cref{cor- expression of H0KerIGf}, we know that the middle vertical sequence is a short exact sequence. Hence, by the $(3\times 3)$-Lemma \cite[Theorem 12]{Bourn2001}, the left-hand vertical sequence, in the above diagram, is a short exact sequence and therefore, we have our conclusion
\[
	\H_{2}(C,I)_{\G}=\H_1(I(\G(C)))\cong \frac{\Ker(I(\partial^1_0))\cap \Ker(I(\partial^1_1))}{I(\partial_0^2)(\Ker(I(\partial^2_1))\cap \Ker(I(\partial^2_2)))}.
\]

\section{Non-additive left derived functors via chain resolutions}\label{sec - non additive derived functors and Cech resolution}

The previous sections concern derived functors of any functor using simplicial methods. There exists also a notion of derived functors for additive functors via chain methods (see e.g.\ \cite{Weibel}). In~\cite{MC-FR-TVdL}, the authors explore the concept of left-derived functors for functors between non-abelian categories, using chain resolutions. They provide clear conditions which are automatically satisfied if we consider additive functors from an abelian category with enough projectives into another abelian category. Moreover, they compare their approach with classical derived functors computed via simplicial methods and they obtain:
\begin{theorem}~\cite[Theorem 11.4.]{MC-FR-TVdL}
	Let $\C$ be a semi-abelian variety for which Condition \P\ holds, with comonad $\G$ induced by the forgetful adjunction to $\Set$. Let $\E$ be a semi-abelian category and $F\colon \C\to \E$ a protoadditive functor~\cite{Everaert-Gran-TT,EG-honfg} which preserves proper morphisms (morphisms with cokernel-kernel factorization) and binary coproducts. Then $\Left_n(F)$ is naturally isomorphic to $\H_{n+1}(-,F)_{\G}$, the $(n+1)$th simplicially derived functor of~$F$ in the sense of Barr--Beck~\cite{Barr-Beck,EverVdL2}.\noproof
\end{theorem}
In the above statement, $\Left_n(F)\coloneq \H_n(F(C(X)))$ where $C(X)$ is a chosen projective resolution of an object $X\in \C$ and, given a semi-abelian category $\C$, Condition \P\ can be reformulated as follows: given any split short exact sequence
\[
	\xymatrix{0 \ar[r] & K \ar@{{ |>}->}[r]^-k & X \ar@{-{ >>}}@<.5ex>[r]^-{f} & Y \ar[r] \ar@{{ >}->}@<.5ex>[l]^-{s} & 0}
\]
in $\C$, $K$ is projective as soon as $X$ is projective. For instance, any Schreier variety trivially satisfies \P~\cite[Example 3.2.6.]{MC-FR-TVdL}) as well as any additive category with finite limits~\cite[Example 3.2.5]{MC-FR-TVdL}.

It is important to mention that the proof of this last theorem is based on another and equivalent definition of Moore complexes: let $\mathbb{A}$ be a (semi-)simplicial object in a pointed category with pullbacks $\C$ and $n\geq 1$. From $\left(\partial_i^n\colon A_n\to A_{n-1}\right)_{0\leq i\leq n}$, we can either, as done in this article, take $\left(\partial_i^n\colon A_n\to A_{n-1}\right)_{1\leq i\leq n}$ to define $\N_n(\mathbb{A})$ and where $d_n^{\N(\mathbb{A})}$ is the restriction of $\partial_0^n$ to the kernels. Or, as done in~\cite{MC-FR-TVdL}, we can take $\left(\partial_i^n\colon A_n\to A_{n-1}\right)_{0\leq i\leq n-1}$ to define $\N_n(\mathbb{A})$ and the restriction of $\partial_n^n$ to the kernels will determine $d_n^{\N(\mathbb{A})}$. Actually, we can prove that the two definitions will induced isomorphic derived functors.

Therefore, combining the foregoing theorem with~\cref{thm - theorem Barr Beck n=2 same as Cech n=1}, we can state that: 
\begin{theorem}\label{them - Cech derived functor and non-additive left derived functors}
	Let $\C$ be a semi-abelian variety for which Condition \P\ holds, with comonad $\G$ induced by the forgetful adjunction to $\Set$. Let $\D$ be a subvariety of $\C$ and $I\colon \C\to \D$ the associated reflector. If $I$ is a protoadditive functor which preserves proper morphisms then $\Le_1(I)$ is naturally isomorphic to the non-additive left derived functor $\Left_1(I)$ in the sense of~\cite{MC-FR-TVdL}.\noproof
\end{theorem}

A concrete application of this new theorem is given by the connected components functor $\pi_0\colon \XMod(\V)\to \V$ where $\V$ is a semi-abelian variety satisfying \P\ and where $\XMod(\V)$ denotes the category of internal crossed modules in $\V$~\cite{Janelidze} (see~\cite{MC} for further details):

\begin{corollary}\label{cor - Cech derived functor pi0}
	With the above notations and assumptions, $\Le_1\left(\pi_0\right)$ is naturally isomorphic to the non-additive left derived functor $\Left_1\left(\pi_0\right)$.\noproof
\end{corollary}

As explained in~\cite[Section 5.2.]{MC}, the foregoing corollary can be extended as follows: starting with a semi-abelian variety $\V$ satisfying \P, we have the iterated adjunctions
\[
\xymatrix{
	\V \ar@<-1.1ex>[r]_-{D}^-{\perp}& \XMod(\V) \ar@<-1.1ex>[r]_-{D}^-{\perp} \ar@<-1.1ex>[l]_-{\pi_{0,\V}} & \XMod^2(\V) \ar@<-1.1ex>[l]_-{\pi_{0,\XMod(\V)}}
}
\]
where $\XMod^2(\V)\coloneq \XMod(\XMod(\V))$. We write $\pi_{0,\V}^2\coloneq \pi_{0,\V}\pi_{0,\XMod(\V)}$ for the composed reflector and inductively define $\pi_{0,\V}^n\coloneq \pi_{0,\V}\pi_{0,\V}^{n-1}$ for $n\geq 1$:
\begin{corollary}\label{cor - Cech derived functor pi0 comp}
	With the above notations and assumptions, let $m,n\in \NN$ such that $m\geq 1$. If we consider the compositions
	\[
	\pi_{0,\XMod^n(\V)}^m\colon \XMod^{m+n}(\V)\to \XMod^{n}(\V)
	\]
	of $\pi_0$, then $\Le_1\left(\pi_{0,\XMod^n(\V)}^m\right)$ is naturally isomorphic to the non-additive left derived functor $\Left_1\left(\pi_{0,\XMod^n(\V)}^m\right)$.\noproof
\end{corollary}

\section*{Acknowledgments}
Many thanks to Tim Van~der~Linden who brought us together to collaborate on this work and also for his feedback on the early drafts of the article.
	

\end{document}